\documentclass[reqno,11pt]{amsart}
\makeatletter
\let\origsection=\section 
\def\section{\@ifstar{\origsection*}{\mysection}}
\def\mysection{\@startsection{section}{1}\z@{.7\linespacing\@plus\linespacing}{.5\linespacing}{\normalfont\scshape\centering\S}}
\makeatother

\usepackage{geometry}
\usepackage[english]{babel}

\usepackage{latexsym}
\usepackage{amsfonts}
\usepackage{amsmath}
\usepackage{url}
\usepackage{amsthm}
\usepackage{amssymb}
\usepackage[utf8]{inputenc}
\usepackage[only,ninrm,elvrm,twlrm,fivrm,sixrm,sevrm,egtrm,egtit,tenrm,tenit,twlit,frtnrm]{rawfonts}
\usepackage{indentfirst}
\usepackage{enumerate}
\usepackage{verbatim}
\usepackage{color}

\usepackage{xcolor}
\usepackage[backref=page]{hyperref}
\hypersetup{%
    colorlinks,
    linkcolor={red!60!black},
    citecolor={green!60!black},
    urlcolor={blue!60!black}
}

\usepackage[abbrev,msc-links,backrefs]{amsrefs}

\usepackage[normalem]{ulem} 

\usepackage{graphicx,psfrag} 

\sloppypar

\newcommand{\sP}{\mathcal{P}}
\newcommand{\sQ}{\mathcal{Q}}
\newcommand{\sR}{\mathcal{R}}

\newcommand{\Min}[1]{\min \{#1\}}
\newcommand{\sC}{\mathcal{C}}
\newcommand{\bs}{\backslash}
\newcommand{\set}[1]{\{#1\}}
\newcommand{\sB}{\mathcal{B}}

\newtheorem{theorem}{Theorem}[section]
\newtheorem{corollary}{Corollary}[section]

\newtheorem{lemma}{Lemma}[section]
\newtheorem{conjecture}{Conjecture}[section]

\newtheorem*{conjecture-linial}{Linial's Conjecture~\cite{Linial1981}}
\newtheorem*{conjecture-dual-linial}{Linial's Dual Conjecture~\cite{Linial1981}}
\newtheorem*{conjecture-berge}{Berge's Conjecture~\cite{Berge1982}}
\newtheorem*{conjecture-ahh}{Aharoni-Hartman-Hoffman's Conjecture~\cite{AharoniEtAl1985}}

\newcounter{factcount}
\newcounter{casecount}
\newcounter{casefactcount}
\usepackage{etoolbox}
\AtBeginEnvironment{proof}{\setcounter{factcount}{0}}
\AtBeginEnvironment{case}{\setcounter{casefactcount}{0}}

\newcommand{\Compl}[1]{\overline{#1}}

\begin{document}

\title[Berge's and Aharoni-Hartman-Hoffman's Conjecture for in-semicomplete digraphs]{Berge's Conjecture and Aharoni-Hartman-Hoffman's Conjecture for locally in-semicomplete digraphs}

\author[
Sambinelli
\and Negri Lintzmayer
\and Nunes da Silva
\and Lee
]{Maycon Sambinelli \and Carla Negri Lintzmayer \and Cândida Nunes da Silva \and Orlando Lee}

\address{Institute of Computing, University of Campinas, Campinas, São Paulo, Brazil}
\email{\{msambinelli|carlanl|lee\}@ic.unicamp.br}

\address{Department of Computing, Federal University of São Carlos, Sorocaba, São Paulo, Brazil}
\email{candida@ufscar.br}

\thanks{%
M. Sambinelli is supported by CNPq (Proc. 141216/2016-6),
C. N. Lintzmayer by FAPESP (Proc. 2016/14132-4), and
O. Lee by CNPq (Proc. 311373/2015-1) and FAPESP (Proc. 2015/11937-9).
}

\begin{abstract}
  Let $k$ be a positive integer and let $D$ be a digraph.
  A \emph{path partition} $\sP$ of $D$ is a set of vertex-disjoint paths which covers $V(D)$. 
  Its \emph{$k$-norm} is defined as $\sum_{P \in \sP} \Min{|V(P)|, k}$.
  A path partition is \emph{$k$-optimal} if its $k$-norm is minimum among all path partitions of $D$.
  A \textit{partial $k$-coloring} is a collection of $k$ disjoint stable sets. 
  A partial $k$-coloring $\sC$ is \emph{orthogonal} to a path partition $\sP$ if each path $P \in \sP$ meets $\min\{|P|,k\}$ distinct sets of $\sC$.
  Berge (1982) conjectured that every $k$-optimal path partition of $D$ has a partial $k$-coloring orthogonal to it.
  A \emph{(path) $k$-pack} of $D$ is a collection of at most $k$ vertex-disjoint paths in $D$.
  Its \emph{weight} is the number of vertices it covers.
  A $k$-pack is \emph{optimal} if its weight is maximum among all $k$-packs of $D$.
  A \emph{coloring} of $D$ is a partition of $V(D)$ into stable sets.
  A $k$-pack $\sP$ is \emph{orthogonal} to a coloring $\sC$ if each set $C \in \sC$ meets $\Min{|C|, k}$ paths of $\sP$.
  Aharoni, Hartman and Hoffman (1985) conjectured that every optimal $k$-pack of $D$ has a coloring orthogonal to it.
  A digraph $D$ is \emph{semicomplete} if every pair of distinct vertices of $D$ is adjacent. 
  A digraph $D$ is \emph{locally in-semicomplete} if, for every vertex $v \in V(D)$, the in-neighborhood of $v$ induces a semicomplete digraph.
  \emph{Locally out-semicomplete} digraphs are defined similarly.
  In this paper, we prove Berge's and Aharoni-Hartman-Hoffman's Conjectures for locally in/out-semicomplete digraphs.
\end{abstract}

\maketitle


\section{Introduction}

The digraphs considered in this text do not contain loops or parallel arcs, but may contain cycles of length two. 
Let $D$ be a digraph.
We denote the vertex set of $D$ by $V(D)$ and its arc set by $A(D)$. 
If $u$ and $v$ are vertices of $D$, then we denote the arc with tail in $u$ and head in $v$ by $uv$.
Vertices $u$ and $v$ are \textit{adjacent} in~$D$ if $uv \in A(D)$ or $vu \in A(D)$; otherwise they are \textit{nonadjacent}.
The \emph{neighborhood}, \emph{in-neighborhood}, and \emph{out-neighborhood} of a vertex $v \in V(D)$ are the sets $\{u \in V(D) \colon uv \in A(D) \text{ or } vu \in A(D)\}$, $\{u \in V(D) \colon uv \in A(D)\}$, and $\{u \in V(D) \colon vu \in A(D)\}$, respectively.

A \textit{path} in $D$ is a nonempty sequence of distinct vertices $P = v_1v_2 \ldots v_\ell$ such that $v_iv_{i+1} \in A(D)$ for $1 \leq i < \ell$.
We define $V(P) = \{v_1, v_2, \ldots, v_\ell\}$ and $e(P) = v_\ell$.
The \emph{order} of $P$, denoted by $|P|$, is equal to $\ell$ and a path is \emph{trivial} if its order is one.
We denote the order of a longest path in $D$ by $\lambda(D)$.
For a set $\sP$ of vertex-disjoint paths of $D$, we define $V(\sP) = \cup_{P \in \sP} V(P)$.

A set $S$ of vertices of $D$ is \textit{stable} if all of its vertices are pairwise nonadjacent.
The \emph{stability number} of $D$, denoted by $\alpha(D)$, is equal to the cardinality of a maximum stable set of $D$.
A \textit{path partition} of $D$ is a set of vertex-disjoint paths of $D$ which covers $V(D)$. 
A path partition $\sP$ of $D$ is \emph{optimal} if $|\sP|$ is minimum among all possible path partitions of $D$.
The cardinality of an optimal path partition of $D$ is denoted by $\pi(D)$.
In 1950, Dilworth~\cite{Dilworth1950} proved the following result.

\begin{theorem}[Dilworth~\cite{Dilworth1950}]\label{the:dilworth}
  For every transitive acyclic digraph $D$, we have $\pi(D) = \alpha(D)$.
\end{theorem}

Note that this equality is not valid for digraphs in general; for example, if $D$ is a directed cycle with 5 vertices, then $\pi(D) = 1$ and $\alpha(D) = 2$. 
However, Gallai and Milgram~\cite{GallaiMilgram1960} proved that the following inequality holds for arbitrary digraphs.

\begin{theorem}[Gallai-Milgram~\cite{GallaiMilgram1960}]\label{the:gallai-milgram}
  For every digraph $D$, we have $\pi(D) \leq \alpha(D)$.
\end{theorem}

Let $k$ be a positive integer and let $D$ be a digraph.
The \textit{$k$-norm} of a path partition $\sP$ of $D$ is defined as $\sum_{P \in \sP} \min\{|P|,k\}$ and denoted by $|\sP|_k$.
A path partition $\sP$ of $D$ is \textit{$k$-optimal} if  $|\sP|_k$ is minimum among all possible path partitions of $D$.
The $k$-norm of a $k$-optimal path partition of $D$ is denoted by $\pi_k(D)$. 
A \textit{partial $k$-coloring} $\sC$ of $D$ is a collection of $k$ disjoint stable sets of $D$ called \textit{color classes} (empty color classes are allowed). 
The \textit{weight} of a partial $k$-coloring $\sC$, denoted by $||\sC||$, is defined as $\sum_{C \in \sC} |C|$. 
A partial $k$-coloring $\sC$ of $D$ is \emph{optimal} if $||\sC||$ is maximum among all possible partial $k$-colorings of $D$.
The weight of an optimal partial $k$-coloring of $D$ is denoted by $\alpha_k(D)$. 
In 1976, Greene and Kleitman~\cite{GreeneKleitman1976} proved the following result.

\begin{theorem}[Greene-Kleitman~\cite{GreeneKleitman1976}]\label{the:greene-kleitman}
  For every transitive acyclic digraph $D$ and every positive integer $k$, we have $\pi_k(D) = \alpha_k(D)$.
\end{theorem}

Since $\pi(D) = \pi_1(D)$ and $\alpha(D) = \alpha_1(D)$, Theorem~\ref{the:dilworth} is the particular case of Theorem~\ref{the:greene-kleitman} in which $k = 1$.
In 1981, Linial~\cite{Linial1981} conjectured that Theorem~\ref{the:greene-kleitman} can be extended to arbitrary digraphs in the same way that Theorem~\ref{the:gallai-milgram} extends Theorem~\ref{the:dilworth}.

\begin{conjecture}[Linial~\cite{Linial1981}]\label{co:linial}
  For every digraph $D$ and every positive integer $k$, we have $\pi_k(D) \leq \alpha_k(D)$.
\end{conjecture}

In an attempt to unify Theorem~\ref{the:gallai-milgram} and a result proved independently by Gallai~\cite{Gallai1968} and Roy~\cite{Roy1967} (Theorem~\ref{the:gallai-roy}), Berge proposed the following conjecture, which is a strengthening of Conjecture~\ref{co:linial}.
A path partition $\sP$ and a partial $k$-coloring $\sC$ are \emph{orthogonal} if each path $P \in \sP$ meets $\min\{|P|,k\}$ distinct color classes of $\sC$ (we also say that $\sP$ is orthogonal to $\sC$ and vice-versa).

\begin{conjecture-berge}
  Let $D$ be a digraph and let $k$ be a positive integer.
  If $\sP$ is a $k$-optimal path partition of $D$, then there exists a partial $k$-coloring of $D$ orthogonal to $\sP$.
\end{conjecture-berge}

Berge's Conjecture remains open,
but we know it holds for $k = 1$~\cite{Linial1978},
$k = 2$~\cite{BergerHartman2008}, when $\lambda(D) = 3$~\cite{Berge1982},
when the $k$-optimal path partition has only paths of order at most $k$~\cite{Berge1982} or if it has only paths of order at least $k$~\cite{AharoniHartman1993},
acyclic digraphs~\cite{Cameron1986,AharoniEtAl1985}, digraphs where all directed cycles are pairwise vertex-disjoint~\cite{Sridharan1993},
bipartite digraphs~\cite{Berge1982}, digraphs containing a Hamiltonian path~\cite{Berge1982},
and $k \geq \lambda(D) - 3$~\cite{Herskovics2016}.

Now we exchange the roles of paths and stable sets in the concepts discussed so far and present some similar results.
Let $D$ be a digraph.
A \emph{coloring} of $D$ is a partition of $V(D)$ into stable sets called \textit{color classes}.
A coloring $\sC$ of $D$ is \emph{optimal} if $|\sC|$ is minimum among all possible colorings of $D$.
The \textit{chromatic number} of $D$, denoted by $\chi(D)$, is the cardinality of an optimal coloring of $D$.
Mirsky~\cite{Mirsky1971} proved the following dual of Theorem~\ref{the:dilworth}.

\begin{theorem}[Mirsky~\cite{Mirsky1971}]\label{the:mirsky}
 For every transitive acyclic digraph $D$, we have $\chi(D)=\lambda(D)$.
\end{theorem}

Similarly to Theorem~\ref{the:gallai-milgram},  Gallai~\cite{Gallai1968} and Roy~\cite{Roy1967}, independently, proved the following result.
\begin{theorem}[Gallai-Roy~\cite{Gallai1968,Roy1967}]\label{the:gallai-roy}
  For every digraph $D$, we have $\chi(D) \leq \lambda(D)$.
\end{theorem}

Let $k$ be a positive integer.
The \textit{$k$-norm} of a coloring $\sC$, denoted by $|\sC|_k$, is $\sum_{C \in \sC} \Min{|C|, k}$.
A coloring $\sC$ of a digraph $D$ is \textit{$k$-optimal} if $|\sC|_k$ is minimum among all possible path partitions of $D$.
The $k$-norm of a $k$-optimal coloring of a digraph $D$ is denoted by $\chi_k(D)$. 
A \emph{(path) $k$-pack} $\sP$ is a set of at most $k$ vertex-disjoint paths of a digraph $D$.
The \emph{weight} of a $k$-pack $\sP$, denoted by $||\sP||$, is defined as $|V(\sP)|$ (i.e., the number of vertices $\sP$ covers).
A $k$-pack $\sP$ of a digraph $D$ is \textit{optimal} if $||\sP||$ is maximum among all possible $k$-packs of $D$.
The weight of an optimal $k$-pack of a digraph $D$ is denoted by $\lambda_k(D)$.
Note that $\chi(D) = \chi_1(D)$ and $\lambda(D) = \lambda_1(D)$.
Greene~\cite{Greene1976} proved the following theorem for transitive acyclic digraphs.

\begin{theorem}[Greene~\cite{Greene1976}]\label{the:greene}
  For every transitive acyclic digraph $D$ and every positive integer $k$, we have $\chi_k(D) = \lambda_k(D)$.
\end{theorem}

Linial~\cite{Linial1981} proposed Conjecture~\ref{co:dual-linial} for arbitrary digraphs.

\begin{conjecture}[Linial~\cite{Linial1981}]\label{co:dual-linial}
  For every digraph $D$ and every positive integer $k$, we have $\chi_k(D) \leq \lambda_k(D)$.
\end{conjecture}

Note that Theorems~\ref{the:mirsky}, \ref{the:gallai-roy}, \ref{the:greene}, and Conjecture~\ref{co:dual-linial} can be seen as dual versions of Theorems~\ref{the:dilworth}, \ref{the:gallai-milgram}, \ref{the:greene-kleitman}, and Conjecture~\ref{co:linial}, respectively, where the roles of paths and stable sets are exchanged.
Therefore, it is natural to ask if there exists a dual version of Berge's Conjecture.
By exchanging the roles of paths and stable sets, we end up with the following definition of orthogonality. 
A coloring $\sC$ and a $k$-pack $\sP$ are \textit{orthogonal} if each color class $C \in \sC$ meets $\min\{|C|,k\}$ distinct paths of $\sP$ (we also say that $\sC$ is orthogonal to $\sP$ and vice-versa).
The natural dual version of Berge's Conjecture states that a $k$-optimal coloring of a digraph $D$, for some positive integer $k$, has a $k$-pack of $D$ orthogonal to it.
This statement is a strengthening of Conjecture~\ref{co:dual-linial}, however it is false.
For example, take $D$ defined as $V(D)$ $=$ $\{v_1, v_2, v_3, v_4, v_5\}$ and $A(D) = \{v_1v_2, v_1v_5, v_3v_2, v_3v_4, v_5v_4\}$, and take $\sC = \{\{v_1, v_4\}, \{v_2, v_5\}, \{v_3\}\}$.
As an alternative strengthening of Conjecture~\ref{co:dual-linial}, Aharoni, Hartman, and Hoffman~\cite{AharoniEtAl1985} gave the following conjecture.

\begin{conjecture-ahh}
  Let $D$ be a digraph and let $k$ be a positive integer.
  If $\sP$ is an optimal $k$-pack of $D$, then
  there exists a coloring of $D$ orthogonal to $\sP$.
\end{conjecture-ahh}

This conjecture remains open, but we know it holds for $k = 1$~\cite{Gallai1968,Roy1967}, $k \geq \pi(D)$~\cite{HartmanEtAl1994}, when the optimal $k$-pack has at least one trivial path~\cite{HartmanEtAl1994}, bipartite digraphs~\cite{HartmanEtAl1994}, and acyclic digraphs~\cite{AharoniEtAl1985}.

A digraph is \emph{semicomplete} if all its vertices are pairwise adjacent.
A digraph $D$ is \emph{(locally) in-semicomplete} (respectively, \emph{out-semicomplete}) if, for every vertex $v \in V(D)$, the in-neighborhood (respectively, out-neighborhood) of $v$ induces a semicomplete digraph. 
One important characterization of in-semicomplete digraphs which we use throughout the text is the following.

\begin{theorem}[\cite{Bang-JensenGutin2008}]\label{the:path-mergeability}
  A digraph $D$ is in-semicomplete if, and only if, for every vertex $v$ and every pair of internally vertex-disjoint paths $P$ and $Q$ such that $v = e(P) = e(Q)$, there exists a path $R$ such that $V(R) = V(P) \cup V(Q)$ and $e(R) = v$.
\end{theorem} 

In-semicomplete digraphs generalize semicomplete digraphs, which in turn generalize \emph{tournaments}.
We refer the reader to the book by Bang-Jensen and Gutin~\cite{Bang-JensenGutin2008} for further information on this class of digraphs. 
We just would like to remark that there are a few results and problems in
literature considering in-semicomplete digraphs, such as Bondy's Conjecture and Laborde, Payan, and Xuong's Conjecture, presented next.
The former states that for every digraph $D$ and every choice of positive integers $\lambda_1$, $\lambda_2$ such that $\lambda(D) = \lambda_1+\lambda_2$, there exists a partition of $D$ into two digraphs $D_1$ and $D_2$ such that $\lambda(D_i) = \lambda_i$, for $i=1,2$.
The latter states that in every digraph there exists a maximal stable set that intersects every longest path.
These conjectures, still open for arbitrary digraphs, were proved for in-semicomplete digraphs by Bang-Jensen \textit{et al.}~\cite{2006-bang-jensen-etal} and Galeana-S\'{a}nchez and G\'{o}mez~\cite{2008-sanchez-gomez}, respectively.
In this paper, we prove Berge's Conjecture and Aharoni-Hartman-Hoffman's Conjecture for in-semicomplete and out-semicomplete digraphs.

\section{Results for Berge's Conjecture}

Given a path $P$ and a positive integer $k$, if $|P| > k$ then we say $P$ is \textit{$k$-long}, otherwise we say it is \textit{$k$-short}.
For a set $\sP$ of vertex-disjoint paths of a digraph $D$ and a positive integer $k$, we define $e(\sP) = \{e(P) \colon P \in \sP\}$, $\sP^{>k} = \{P \in \sP \colon |P| > k\}$, and $\sP^{\leq k} = \{P \in \sP \colon |P| \leq k\}$.
To simplify notation, given a set $S$ and an element $x$, we denote by $S + x$ the union $S \cup \{x\}$ and by $S - x$ the difference $S \bs \{x\}$.

Next lemma shows that it is possible to convert one path partition $\sP$ into another path partition $\sQ$ whose $k$-norm is either smaller than $\sP$'s or it is the same as $\sP$'s but $e(\sQ)$ is a stable set.

\begin{lemma}\label{lem:partition-stable}
  Let $\sP$ be a path partition of an in-semicomplete digraph $D$ and let $k$ be a positive integer.
  Then there exists a path partition $\sQ$ of $D$ such that one of the following conditions holds:
  \begin{enumerate}[\rm (i)]
  \item $|\sQ|_k < |\sP|_k$ and $e(\sQ) \subset e(\sP)$;
  \item $|\sQ|_k = |\sP|_k$, $e(\sQ) \subseteq e(\sP)$, $e(\sQ)$ is stable, and every partial $k$-coloring of $D$ orthogonal to $\sQ$ is also orthogonal to $\sP$.
  \end{enumerate}
\end{lemma}
\begin{proof}
  If $e(\sP)$ is stable, then $\sQ = \sP$ satisfies case (ii) and the result follows.
  Thus, we may assume that $e(\sP)$ is not stable and, therefore, there exists a pair of vertices $u$ and $v$ in $e(\sP)$ such that $uv \in A(D)$.
  Let $P_1$ and $P_2$ be the paths in $\sP$ such that $e(P_1) = u$ and $e(P_2) = v$.
  By Theorem~\ref{the:path-mergeability}, there exists a path $Q$ in $D$ such that $V(Q) = V(P_1) \cup V(P_2)$ and $e(Q) = v$.
  Let $\sQ$ be the path partition of $D$ defined as $\sP - P_1 - P_2 + Q$.
  Note that $e(\sQ) \subset e(\sP)$.

  Suppose first that at least one of $P_1$ and $P_2$ is a $k$-long path.
  Hence, $Q$ is $k$-long and
  \begin{align}
    |\sQ|_k &= \sum_{P \in \sQ} \Min{|P|, k} = \sum_{P \in \sQ - Q} \Min{|P|, k} + \Min{|Q|, k} = \sum_{P \in \sQ - Q} \Min{|P|, k} + k\nonumber\\
            &< \sum_{P \in \sP - P_1 - P_2} \Min{|P|, k} + \Min{|P_1|, k} + \Min{|P_2|, k} = \sum_{P \in \sP} \Min{|P|, k} = |\sP|_k \nonumber \enspace,
  \end{align}
  where the inequality follows because $\Min{|P_1|, k} + \Min{|P_2|, k}$ is at least $k + 1$, since at least one of $P_1$ or $P_2$ is $k$-long.
  Therefore, case (i) holds and we may assume that there is no arc in $A(D)$ for which one of its endpoints is in $e(\sP^{> k})$ and the other is in $e(\sP)$.
  Hence, we have that $P_1$ and $P_2$ are $k$-short paths.

  Note that $|\sP|_k$ $=$ $\sum_{P \in \sP - P_1 - P_2} \Min{|P|, k}$ $+$ $\Min{|P_1|,k}$ $+$ $\Min{|P_2|,k}$ and $|\sQ|_k$ $=$ $\sum_{P \in \sP - P_1 - P_2} \Min{|P|, k}$ $+$ $\Min{|Q|, k}$.
  If $|Q| > k$, that is, if $Q$ is $k$-long, then $|\sQ|_k < |\sP|_k$ and so $\sQ$ satisfies case (i) of the lemma.

  Thus, we may assume $Q$ is $k$-short and, therefore, $|\sQ|_k = |\sP|_k$.
  The previous argument shows that if there exists an arc of $D$ connecting two vertices of $e(\sP)$, then both of them must be ends of $k$-short paths of $\sP$.
  From $\sQ$ we show how to find a new path partition of $D$ which satisfies either (i) or (ii).

  First suppose that there exists a partial $k$-coloring $\sC$ orthogonal to $\sQ$ and let $P$ be a path in $\sQ - Q = \sP - P_1 - P_2$.
  Hence, $P$ meets $\Min{|P|, k}$ color classes of $\sC$, since $\sC$ is orthogonal to $\sQ$.
  In order to prove that $\sC$ is also orthogonal to $\sP$, it remains to show that $P_i$ meets $\Min{|P_i|, k} = |P_i|$ color classes for $i \in \set{1, 2}$.
  Since $Q$ is $k$-short, we know that each of its vertices meets a different color class of $\sC$.
  Therefore, every vertex in $P_i$, for $i \in \{1, 2\}$, also meets a distinct color class of $\sC$, after all $V(Q) = V(P_1) \cup V(P_2)$, and so $\sC$ is indeed orthogonal to $\sP$.
  So, we assume there exists no partial $k$-coloring orthogonal to $\sQ$.

  If $e(\sQ^{\leq k})$ is stable, then $e(\sQ)$ is stable and case (ii) holds.
  Thus, we may assume that $e(\sQ^{\leq k})$ is not stable and the remaining proof is by induction on the number $\ell$ of $k$-short paths in $\sP$.

  Since $e(\sP^{\leq k})$ is not stable, we have $\ell \geq 2$.
  If $\ell = 2$, then $\sP^{\leq k} = \{P_1, P_2\}$.
  Therefore, $e(\sQ^{\leq k}) = e(Q)$ and so $\sQ$ satisfies case (ii).
  Now suppose $\ell > 2$.
  By the induction hypothesis applied to $\sQ$, there exists a path partition $\sQ'$ of $D$ such that case (i) or (ii) holds.
  If (i) holds, then $e(\sQ') \subset e(\sQ)$ and $|\sQ'|_k < |\sQ|_k$.
  By construction of $\sQ$, this also means that $e(\sQ') \subset e(\sP)$ and $|\sQ'|_k < |\sP|_k$, and so case (i) holds for $\sP$.
  Otherwise (ii) holds for $\sQ'$, that is, $|\sQ'|_k = |\sQ|_k = |\sP|_k$, $e(\sQ') \subseteq e(\sQ) \subset e(\sP)$, $e(\sQ')$ is stable, and every partial $k$-coloring orthogonal to $\sQ'$ is also orthogonal to $\sQ$ and, therefore, to $\sP$.
  Thus, (ii) holds for $\sP$.
\end{proof}

Given a path $P = v_1v_2 \ldots v_\ell$, we write $v_iP = v_i v_{i + 1} \ldots v_\ell$, $Pv_j = v_1v_2 \ldots v_j$, and $v_iPv_j = v_i v_{i + 1} \ldots v_j$ to denote the appropriate subpath of $P$.
Also, using the definitions of $\sP^{>k}$ and $\sP^{\leq k}$ given above, note that the $k$-norm of a path partition $\sP$ can equivalently be defined as $k|\sP^{>k}| + |V(\sP^{\leq k})|$.

Next theorem is the first main result of this paper.
It shows that any path partition of an in-semicomplete digraph either has a partial $k$-coloring orthogonal to it or can be turned into a new path partition with smaller $k$-norm.
Corollaries~\ref{co:berge-in-semicomplete} and \ref{co:berge-out-semicomplete} state the meaning of such result for Berge's Conjecture.

\begin{theorem}
\label{th:bondy-in-semicomplete}
  Let $D$ be an in-semicomplete digraph, let $k$ be a positive integer, and let $\sP$ be a path partition of $D$.
  Then there exists
  \begin{enumerate}[\rm (i)]
  \item a partial $k$-coloring of $D$ orthogonal to $\sP$; or
  \item a path partition $\sQ$ of $D$ such that $|\sQ|_k < |\sP|_k$  and $e(\sQ) \subset e(\sP)$.
  \end{enumerate}
\end{theorem}
\begin{proof}
  By Lemma~\ref{lem:partition-stable}, there exists a path partition $\sQ$ of $D$ such that either (a) $|\sQ|_k < |\sP|_k$ and $e(\sQ) \subset e(\sP)$ or (b) $|\sQ|_k = |\sP|_k$, $e(\sQ) \subseteq e(\sP)$, $e(\sQ)$ is stable, and every partial $k$-coloring orthogonal to $\sQ$ is also orthogonal to $\sP$.
  If (a) holds, then case (ii) holds directly.
  Therefore, we may assume that (b) holds.
  Note that this reduces the problem of proving the result for $\sP$ to the problem of proving it for $\sQ$ and, thus, from now on we can only consider $\sQ$.

  The remaining proof follows by induction on $k$.
  If $k = 1$, then $e(\sQ)$ is a partial $1$-coloring orthogonal to $\sQ$ and (i) holds.
  Otherwise, $k > 1$ and let $D' = D[V(D) \bs  e(\sQ)]$ and $\sQ' = \{Qu_{\ell - 1} \colon Q = u_1u_2 \ldots u_\ell \in \sQ\}$.
  Note that $\sQ'$ is a path partition for $D'$.
  By the induction hypothesis applied to $D'$, $\sQ'$, and $k - 1$, we have that there exists (a) a partial $(k-1)$-coloring $\sC$ of $D'$ orthogonal to $\sQ'$ or (b) a path partition $\sR'$ or $D'$ such that  $|\sR'|_{k - 1} < |\sQ'|_{k - 1}$  and $e(\sR') \subset e(\sQ')$.
  If (a) holds, then $\sC + e(\sQ)$ is a partial $k$-coloring orthogonal to $\sQ$ and case (i) holds.

  So we assume that case (b) holds.
  Let $e(\sQ') = \set{v_1, v_2, \ldots, v_\ell}$ and $u_i \in e(\sQ)$ be the sucessor of $v_i$ in its path in $\sQ$.
  Note that each path of $\sR'$ ends at some $v_i$, so name such path as $R'_i$.
  Let $\sR = \{R'_iv_iu_i \colon R'_i \in \sR'^{> k - 1}\} \cup \sR'^{\leq k - 1} \cup (e(\sQ) \bs \{u_i \colon R'_i \in \sR'^{> k - 1}\})$.
  In other words, $\sR$ is built by the extensions of all $(k-1)$-long paths  of $\sR'$, plus all $(k - 1)$-short paths of $\sR'$, plus all single vertices of $e(\sQ)$ which were not used to extend  the $(k-1)$-long paths of $\sR'$.
  It is easy to see that $\sR$ is a path partition for $D$.

  We know that
  $|\sR'|_{k -1} = (k - 1)|\sR'^{> k - 1}| + |V(\sR'^{\leq k - 1})| = k|\sR'^{> k - 1}| + |V(\sR'^{\leq k - 1})| - |\sR'^{> k - 1}|$, by definition.
  By construction of $\sR$, 
  $|\sR|_k = k |\sR'^{> k - 1}| + |V(\sR'^{\leq k - 1})| + |e(\sQ)| - |\sR'^{> k - 1}| = |\sR'|_{k - 1} + |\sQ|$.
  Also, by construction of $\sQ'$,
  $|\sQ'|_{k - 1} = (k - 1)|\sQ^{> k}| + |V(\sQ^{\leq k})| - |\sQ^{\leq k}| = k |\sQ^{> k}| + |V(\sQ^{\leq k})| - |\sQ^{> k}| - |\sQ^{\leq k}| = |\sQ|_k - |\sQ|$.
  Putting everything together, we have
  $|\sR|_k = |\sR'|_{k - 1} + |\sQ| < |\sQ'|_{k - 1} + |\sQ| = |\sQ|_k - |\sQ| + |\sQ| = |\sQ|_k$
  and, therefore, case (ii) holds for $\sQ$.
\end{proof}

\begin{corollary}
\label{co:berge-in-semicomplete}
  If $\sP$ is a $k$-optimal path partition of an in-semicomplete digraph $D$, then there exists a partial $k$-coloring of $D$ orthogonal to $\sP$.
\end{corollary}

\begin{corollary}
\label{co:berge-out-semicomplete}
  If $\sP$ is a $k$-optimal path partition of an out-semicomplete digraph $D$, then there exists a partial $k$-coloring of $D$ orthogonal to $\sP$.
\end{corollary}
\begin{proof}
  The \emph{inverse} of a digraph (path partition) $B$, denoted by $B^-$, is the digraph (path partition) built from $B$ by inverting its arcs, that is, if $uv \in A(B)$, then $vu \in A(B^-)$.
  If $\sQ$ is a path partition of $D$, then $\sQ^-$ is a path partition of $D^-$ with $|\sQ|_k = |\sQ^-|_k$, and vice-versa.
  Thus, we have that $\sP^-$ is $k$-optimal in $D^-$, and since $D^-$ is an in-semicomplete digraph, by Corollary~\ref{co:berge-in-semicomplete} there exists a partial $k$-coloring $\sC$ orthogonal to $\sP^-$.
  Clearly, $\sC$ is also orthogonal to $\sP$.
\end{proof}

\section{Results for Aharoni-Hartman-Hoffman's Conjecture}

Recall that, given a $k$-pack $\sP$, we define $e(\sP) = \{e(P) \colon P \in \sP\}$.
Similarly to the result of Lemma~\ref{lem:partition-stable}, the next lemma shows that it is possible to convert one $k$-pack $\sP$ into another $k$-pack $\sQ$ with the same weight such that $e(\sQ)$ is a stable set.

\begin{lemma}
\label{lem:k-path-stable}
  Let $\sP$ be a $k$-pack of an in-semicomplete digraph $D$.
  Then there exists a $k$-pack $\sQ$ of $D$ such that $||\sQ|| = ||\sP||$, $e(\sQ) \subseteq e(\sP)$, and $e(\sQ)$ is stable.
\end{lemma}
\begin{proof}
  The proof is by induction on $\ell = |\sP|$.
  If $e(\sP)$ is stable, then $\sQ = \sP$ satisfies the lemma's conclusion and the result follows.
  Thus, we may assume $e(\sP)$ is not stable.
  Let $u$ and $v$ in $e(\sP)$ such that $uv \in A(D)$.
  Let $P_1$ and $P_2$ be the paths in $\sP$ which end in $u$ and $v$, respectively.
  By Theorem~\ref{the:path-mergeability}, there exists a path $Q$ in $D$ such that $V(Q) = V(P_1) \cup V(P_2)$ and $e(Q) = e(P_2)$.
  Let $\sQ$ be the $k$-pack of $D$ defined as $\sP - P_1 - P_2 + Q$.
  By the induction hypothesis, there exists a $k$-pack $\sR$ such that $||\sR|| = ||\sQ||$, $e(\sR) \subseteq e(\sQ)$, and $e(\sR)$ is a stable set.
  By construction, $||\sQ|| = ||\sP||$ and $e(\sQ) \subset e(\sP)$, so the result follows directly.
\end{proof}

Next theorem is another main result of this paper.
It shows that any $k$-pack of an in-semicomplete digraph either has a coloring orthogonal to it or can be turned into a $k$-pack with larger weight.
Corollaries~\ref{co:ahh-in-semicomplete} and \ref{co:ahh-out-semicomplete} state the meaning of such result for Aharoni-Hartman-Hoffman's Conjecture.
To simplify the notation, given a $k$-pack $\sP$ of $D$, we denote by $\Compl{\sP}$ the vertex set $V(D) \bs V(\sP)$.
Recall that, given a path $P = v_1v_2 \ldots v_\ell$, we write $v_iP = v_i v_{i + 1} \ldots v_\ell$, $Pv_j = v_1v_2 \ldots v_j$, and $v_iPv_j = v_i v_{i + 1} \ldots v_j$ to denote the appropriate subpaths.

\begin{theorem}
  Let $D$ be an in-semicomplete digraph, let $k$ be a positive integer, and let $\sP$ be a $k$-pack of $D$.
  Then there exists
  \begin{enumerate}[\rm (i)]
  \item a coloring of $D$ orthogonal to $\sP$; or
  \item a $k$-pack $\sQ$ of $D$ such that $||\sQ|| = ||\sP|| + 1$ and $e(\sQ) \subseteq e(\sP) \cup \Compl{\sP}$.
  \end{enumerate}
\end{theorem}
\begin{proof}
  The proof is by induction on $|\Compl{\sP}|$.
  If $\Compl{\sP} = \varnothing$, then the coloring $\{\{v\} \colon v \in V(D)\}$ is orthogonal to $\sP$ and case (i) holds.
  Thus, we may assume $\Compl{\sP} \neq \varnothing$.
  Let $v$ be a vertex in $\Compl{\sP}$ and let $\sQ = \sP + v$.
  If $|\sQ| \leq k$, then $\sQ$ satisfies case (ii) and the result follows.
  Thus, we have $|\sQ| = k + 1$, since $|\sP| \leq k$.
  If $e(\sQ)$ is not stable, then there exist two paths $P_1$ and $P_2$ in $\sQ$ such that $e(P_1)$ and $e(P_2)$ are adjacent and,
  by Theorem~\ref{the:path-mergeability}, there exists a path $Q$ such that $V(Q) = V(P_1) \cup V(P_2)$ and $e(Q) \in \{e(P_1), e(P_2)\}$.
  Therefore, $\sQ - P_1 - P_2 + Q$ is a $k$-pack which satisfies case (ii).
  Since $v$ was chosen arbitrarily from $\Compl{\sP}$, we have that for any $v \in \Compl{\sP}$, $e(\sP) + v$ is stable.
  In particular, $e(\sP)$ is stable.

  Let $S \subseteq \Compl{\sP}$ be a maximum stable set in $D[\Compl{\sP}]$, let $D'$ be the digraph $D[V(D) \bs (e(\sP) \cup S)]$, and let $\sR$ be the $k$-pack of $D'$ defined as $\{Pu_{\ell - 1} \colon P = u_1u_2 \ldots u_\ell \in \sP\}$.
  Note that $||\sR|| = ||\sP|| - k$ and  $\Compl{\sR} \subset \Compl{\sP}$.
  By the induction hypothesis applied to $D'$ and $\sR$, we have that there exists (a) a coloring $\sC$ of $D'$ orthogonal to $\sR$ or (b) a $k$-pack $\sB$ of $D'$ such that $||\sB|| = ||\sR|| + 1$ and $e(\sB) \subseteq e(\sR) \cup \Compl{\sR}$.
  If (a) holds, then $\sC + (e(\sP) \cup S)$ is a coloring orthogonal to $\sP$ and case (i) holds.
  So we may assume that (b) holds.
  We will show that (ii) holds.
  By Lemma~\ref{lem:k-path-stable}, we may assume that $e(\sB)$ is stable.
  Let $\sB = \sB_1 \cup \sB_2$, where $e(\sB_1) \subseteq e(\sR)$ and $e(\sB_2) \subseteq \Compl{\sR}$.

  Let $e(\sR) = \{v_1, v_2, \ldots, v_\ell\}$ and let $u_i \in e(\sP)$ be the successor of $v_i$ in a path of $\sP$.
  Note that every path of $\sB_1$ ends at some $v_i$, so name such path as $B_i$.
  Let $\sQ_1 = \{B_iv_iu_i \colon B_i \in \sB_1\}$, that is, we built $\sQ_1$ by extending all paths of $\sB_1$.
  Note that $||\sQ_1|| = ||\sB_1|| + |\sB_1|$, $|\sQ_1| = |\sB_1|$, and that $e(\sQ_1) \subseteq e(\sP)$.

  Now we will show that there exists a collection of paths in $\Compl{\sP}$ with weight $||\sB_2|| + |\sB_2|$.
  Let $G$ be the bipartite graph with vertex-set $V(G) = e(\sB_2) \cup S$ and edge-set $E(G) = \{uv \colon u \in e(\sB_2), v \in S, \text{ and } u \text{ and } v \text{ are adjacent in } D\}$.
  We claim that there exists a matching in $G$ which covers $e(\sB_2)$.
  Suppose by contradiction that such matching does not exist.
  By Hall's Theorem, there exists $W \subseteq e(\sB_2)$ such that $|W| > N(W)$, where $N(W) = \{u \in V(G) \colon uv \in E(G) \text{ and } v \in W\}$.
  Note that $W$ is stable in $D[\Compl{\sP}]$, since $W \subseteq e(\sB)$, and no vertex in $W$ is adjacent to a vertex in $S \bs N(W)$ in $D$.
  Therefore, we have that $(S \bs N(W)) \cup W$ is stable in $D[\Compl{\sP}]$ greater than $S$, which contradicts the choice of $S$.
  Hence, there exists a matching $M$ in $G$ which covers $e(\sB_2)$.
  For each $u \in e(\sB_2)$, let $M(u)$ be the vertex of $S$ matched to $u$ by $M$.
  Let $\sB_2 = \{B_1, B_2, \ldots, B_p\}$ and let $B_i = w_1w_2 \ldots w_q$ be a path in $\sB_2$.
  By Theorem~\ref{the:path-mergeability}, there exists a path $Q_i$ such that $V(Q_i) = V(B_i) + M(w_q)$ and $e(Q_i) \in \{w_q, M(w_q)\}$.
  Let $\sQ_2 = \{Q_i \colon B_i \in \sB_2\}$.
  Note that $||\sQ_2|| = ||\sB_2|| + |\sB_2|$, $|\sQ_2| = |\sB_2|$, and that $e(\sQ_2) \subseteq \Compl{\sP}$.

  Let $T = (e(\sP) \cup S) \bs (e(\sQ_1) \cup e(\sQ_2))$, that is, the set of vertices in $e(\sP) \cup \Compl{\sP}$ that are not ends of a path in $\sQ_1 \cup \sQ_2$. Thus,
    $|T| = |e(\sP)| + |S| - |\sQ_1| - |\sQ_2|
        = k + |S| - |\sQ_1| - |\sQ_2| \geq k + 1 - |\sQ_1| - |\sQ_2|$.

  Let $U$ be a set of $k - |\sQ_1| - |\sQ_2|$ vertices in $T$.
  We can see $U$ as a set of trivial paths.
  Finally, let $\sQ$ be the $k$-pack of $D$ defined as $\sQ_1 \cup \sQ_2 \cup U$.
  By the previous remarks it is easy to see that $e(\sQ) \subseteq e(\sP) \cup \Compl{\sP}$.
  At last, we have
  \begin{align*}
    ||\sQ|| &= ||\sQ_1|| + ||\sQ_2|| + ||U|| = ||\sB_1|| + |\sB_1| + ||\sB_2|| + |\sB_2| + ||U|| = ||\sB|| + |\sB| + ||U||\\
            &= ||\sB|| + |\sB| + k - |\sQ_1| - |\sQ_2| = ||\sB|| + |\sB| + k - |\sB_1| - |\sB_2| = ||\sB|| + |\sB| + k - |\sB|\\
              &= ||\sR|| + 1 + k  = ||\sP|| - k + 1 + k = ||\sP|| + 1.
  \end{align*}

  Hence, we conclude that (ii) holds and the result follows.
\end{proof}

\begin{corollary}
\label{co:ahh-in-semicomplete}
  If $\sP$ is an optimal $k$-pack of an in-semicomplete digraph $D$, then there exists a coloring of $D$ orthogonal to $\sP$.
\end{corollary}

\begin{corollary}
\label{co:ahh-out-semicomplete}
  If $\sP$ is an optimal $k$-pack of an out-semicomplete digraph $D$, then there exists a coloring of $D$ orthogonal to $\sP$.
\end{corollary}

\bibliographystyle{amsplain}
\bibliography{bib-data/bibliography.bib}

\end{document}